\documentclass[letterpaper]{article} 
\usepackage{aaai25}  
\usepackage{times}  
\usepackage{helvet}  
\usepackage{courier}  
\usepackage[hyphens]{url}  
\usepackage{graphicx} 
\usepackage{natbib}  
\usepackage{caption} 
\title{\astar for Bounding Shortest Paths in the Graphs of Convex Sets}
\author {
    Kaarthik Sundar \textsuperscript{\rm 1},
    Sivakumar Rathinam \textsuperscript{\rm 2}
}
\affiliations {
    \textsuperscript{\rm 1}Los Alamos Research Laboratory, New Mexico, 88220, USA\\
    \textsuperscript{\rm 2}Texas A\&M University, College Station, TX 77843, USA\\
    kaarthik@lanl.gov, srathinam@tamu.edu
}

\usepackage{cite}
\usepackage{graphicx}
\usepackage{amsmath,amssymb,amsthm}
\usepackage{comment}
\usepackage[linesnumbered,ruled,vlined]{algorithm2e}
\usepackage{booktabs}
\usepackage{multirow,multicol}
\usepackage{tcolorbox}
\usepackage{bm}
\usepackage[font=small]{caption}
\usepackage{subcaption}
\usepackage{booktabs, csvsimple}
\usepackage{xspace}
\newcommand{\astar}{$A^*$\xspace}
\newcommand{\agcs}{$A^*$-$GCS$\xspace}
\newcommand{\bagcs}{\textit{\textbf{A$^*$-{GCS}}}\xspace}
\newcommand{\agcsi}{$A^*$-$GCS_{1}$\xspace}
\newcommand{\agcsf}{$A^*$-$GCS_{\infty}$\xspace}
\SetAlCapNameFnt{\small}
\SetAlCapFnt{\small}

\newtheorem{theorem}{Theorem}
\newtheorem{lemma}{Lemma}

\newtheorem{remark}{Remark}

\begin{document}

\maketitle

\begin{abstract}
We present a novel algorithm that fuses the existing convex-programming based approach with heuristic information to find optimality guarantees for the Shortest Path Problem in the Graph of Convex Sets (SPP-GCS). Our method, inspired by \astar, initiates a best-first-like procedure from a designated subset of vertices and iteratively expands it until further growth is neither possible nor beneficial. Traditionally, obtaining solutions with bounds for an optimization problem involves solving a relaxation, modifying the relaxed solution to a feasible one, and then comparing the two solutions to establish bounds. However, for SPP-GCS, we demonstrate that reversing this process can be more advantageous, especially with Euclidean travel costs. We present numerical results to highlight the advantages of our algorithm over the existing approach in terms of the sizes of the convex programs solved and computation time.
\end{abstract}

\section{Introduction}
The Shortest Path Problem (SPP) is one of the most important and fundamental problems in discrete optimization \cite{LAWLER1979, Korte2018}. Given a graph, the SPP aims to find a path between two vertices in the graph such that the sum of the cost of the edges in the path is minimized. In this paper, we concern ourselves with a generalization of the SPP, recently introduced in \cite{GCS}, where each vertex is associated with a convex set and the cost of the edge joining any two vertices depends on the choice of the points selected from each of the respective convex sets. In this generalization, referred to as the {\it Shortest Path Problem in the Graph of Convex Sets} (SPP-GCS), the objective is to find a path and choose a point from each convex set associated with the vertices in the path such that the sum of the cost of the edges in the path is minimized (Fig. \ref{fig:GCS}). 

SPP-GCS reduces to the standard SPP if the point to be selected from each convex set is given. Also, SPP-GCS reduces to a relatively easy-to-solve convex optimization problem if the path is given. Selecting an optimal point in each set, as well as determining the path, makes the SPP-GCS NP-Hard in the general case \cite{GCS}.

We are interested in SPP-GCS because it naturally models SPPs in the geometric domain and planning problems with neighborhoods, with a variety of applications in motion planning \cite{GCS_science, marcucci2023fast, natarajan2024ixg}, sensor coverage \cite{burdick2024, DUMITRESCU2003}, and hybrid control \cite{GCS}. For example, the SPP in the presence of obstacles in 3D can be posed as an SPP-GCS by decomposing the free space into convex sets and then planning on the corresponding graph of these sets. The sensor coverage problem in 2D \cite{burdick2024}, where the robot needs to find a suitable sequence of neighborhoods to traverse while choosing a point from each neighborhood, can also be formulated as a variant of SPP-GCS.

\begin{figure}
    \centering
\includegraphics[scale=0.4]{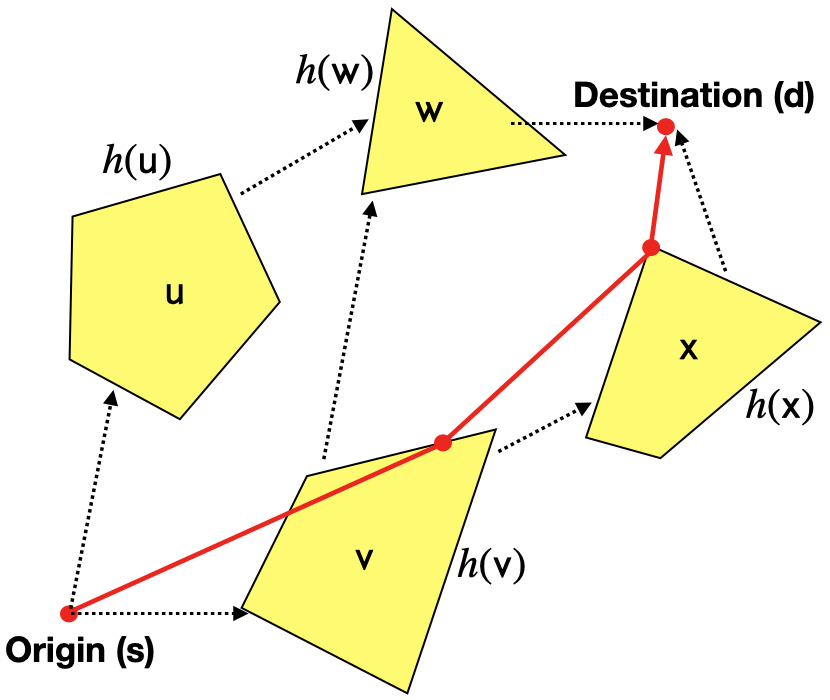}
    \caption{Illustration of the SPP-GCS. The graph has six vertices, with the origin and the destination corresponding to singleton sets. The dotted lines represent the edges in the graph. Each vertex $t$ in the graph is associated with a heuristic value $h(t)$. A feasible path is shown with solid (red) line segments.}
    \label{fig:GCS}
\end{figure}

The importance of formulating a planning problem as an SPP-GCS stems primarily from the following observations: finding optimality guarantees for SPPs in a geometric domain or for planning problems with neighborhoods has been computationally challenging, even in special cases. For example, computing the optimum or a tight lower bound for an SPP in 3D with convex or axis-aligned obstacles remains difficult \cite{mitchell}. While theoretical bounds using polynomial-time approximation schemes are available \cite{mitchell}, we are not aware of any implementations of such schemes with numerical results. Recently, significant progress \cite{GCS} has been made towards addressing this issue where the planning problem is posed as a SPP-GCS and bounds are obtained by solving a tight convex relaxation of the SPP-GCS. Feasible solutions can also be found by either rounding the relaxed solutions or using a sampling-based heuristic (as employed in this paper). Optimality guarantees can then be obtained by comparing the convex relaxation cost (a lower bound) with the length (an upper bound) of a feasible solution. Since the crux of the procedure in \cite{GCS} lies in solving a convex relaxation of the SPP-GCS, this paper is motivated by the need to explore faster methods for solving this relaxation, with the ultimate goal of generating optimality guarantees for planning problems in geometric domains.

One way to develop fast algorithms for an optimization problem is to incorporate heuristic information into the search process. The celebrated $A^*$ algorithm \cite{hart1968aFormalBasis} accomplishes this for the SPP. The objective of this paper is to propose a new algorithm that combines the existing convex-programming based approach in \cite{GCS} with heuristic information to find lower bounds and optimality guarantees for the SPP-GCS. 

We note that there is a parallel development \cite{chia2024gcsforwardheuristicsearch} that aims to incorporate heuristic information into a best-first search process to find optimal solutions for the SPP-GCS. In this process, sub-optimal solutions are pruned based on a domination criterion, which is non-trivial to check for general convex sets and travel costs. However, when the sets are polytopic and the travel costs are linear, algorithms \cite{chia2024gcsforwardheuristicsearch} are available to conservatively prune sub-optimal solutions. This search process and its variants seeks to converge to the optimum from above, which is complementary to the lower-bounding approach we pursue in this paper.

We refer to our algorithm as $A^*$ for Graphs of Convex Sets (\agcs). Just like in $A^*$, we assume that each vertex (or each convex set) has an associated heuristic cost to go from the vertex to the destination (Fig. \ref{fig:GCS}). \agcs follows the spirit of \astar and initiates a best-first-like search procedure from a designated set $S$ (containing the origin) and iteratively builds on it until further growth is neither possible nor beneficial. A key step in \agcs lies in the way we grow $S$ in each iteration. To choose vertices for addition to $S$, we employ a relaxed solution to a generalization of the SPP-GCS that encompasses the vertices in $S$ and their neighbors. In the special case when the convex sets reduce to singletons and the heuristic information is consistent, \agcs reduces to \astar. On the other hand, if heuristic information is not available, \agcs reduces to iteratively solving a convex relaxation of SPP-GCS until the termination conditions are satisfied.  

Unlike \astar, which typically starts its search from a closed set $S$ containing only the origin, \agcs can start its search from any set $S$ as long as it induces a cut — meaning $S$ contains the origin but not the destination. In this paper, we consider two initial choices for $S$ to start the search: one where $S$ only contains the origin, and another where $S$ contains vertices informed by applying \astar to a representative point ($e.g.$ the centroid) from each convex set. One of the key insights we infer from the numerical results with Euclidean travel costs is that the second choice where \agcs starts with the set $S$ informed by \astar provides the best trade-off between achieving good quality bounds and smaller computation times. {In general, \agcs works on relatively smaller convex programs because its vertices are informed by the $A^*$ algorithm, which only explores a subset of vertices in the graph. This leads to reduced computation times compared to solving a convex relaxation on the entire graph.}

Typically, to obtain optimality guarantees for an optimization problem, we solve a relaxation of the problem, modify the relaxed solution to obtain a feasible solution, and then compare the two solutions to establish bounds. {This is the approach followed in \cite{GCS}.} {\it In this paper, we demonstrate that reversing this process can be more effective in obtaining bounds for the SPP-GCS, especially with Euclidean travel costs. Specifically, we first implement \astar using a representative point from each convex set to obtain a feasible solution. We then solve a convex program on subsets of vertices informed by \astar to find a relaxed solution and subsequently compare the solutions.} In hindsight, this approach seems natural; since \astar quickly finds good feasible solutions for the SPP-GCS with a heuristic, it is reasonable to expect that good lower bounds can also be found in the neighborhood of the vertices visited by \astar, based on the strong formulation in \cite{GCS}. 

After presenting our algorithm with theoretical guarantees, we provide extensive computational results on the performance of \agcs for instances derived from mazes, axis-aligned bars and 3D villages to highlight the advantages of our algorithm over the existing approach in terms of the sizes of the convex programs solved and computation time.

\section{Problem Statement}
Let $V$ denote a set of vertices, and $E$ represent a set of directed edges joining vertices in $V$. Let each vertex $v\in V$ be associated with a non-empty, compact convex set $\mathcal{X}_v\subset {\mathbb{R}}^n$. Given vertices $u,v\in V$, the cost of traveling the edge $e=(u,v)$ from $u$ to $v$ depends on the choice of the points in the sets $\mathcal{X}_u$ and $\mathcal{X}_v$. Specifically, given points $x_u\in \mathcal{X}_u$ and $x_v\in \mathcal{X}_v$, let $cost(x_u,x_v)$ denote the travel cost of edge $e=(u,v)$. In this paper, we set $cost(x_u,x_v)$ to be equal to the Euclidean distance between the points $x_u$ and $x_v$, $i.e.$, $cost(x_u,x_v):=\lVert x_u-x_v \rVert _2$.  

    Any path in the graph $(V,E)$ is a sequence of vertices $(v_1,v_2,\cdots,v_k)$ for some positive integer $k$ such that $v_i\in V,$ $i=1,\cdots,k$ and $(v_i,v_{i+1})\in E,$ $i=1,\cdots,k-1$. Given a path $\mathcal{P}:=(v_1,v_2,\cdots,v_k)$ and points $x_{v_i} \in {\mathcal{X}}_{v_i}$, $i=1,\cdots,k$, the cost of traveling $\mathcal{P}$ is defined as $\sum_{i=1}^{k-1} cost(x_{v_i},x_{v_{i+1}})$. Given an origin $s\in V$ and destination $d\in V$, the objective of the SPP-GCS is to find a path $\mathcal{P}:=(v_1,v_2,\cdots,v_k)$ and points $x_{v_i} \in {\mathcal{X}}_{v_i}$, $i=1,\cdots,k$, such that $v_1=s$, $v_k=d$, and the cost of traveling $\mathcal{P}$ is minimized. The optimal cost for SPP-GCS is denoted as \boldmath ${C_{opt}(s,d)}$. \unboldmath
 
\section{Generalization of SPP-GCS and its Relaxation}

We first introduce some notations specific for \agcs. For each vertex $v\in V$, let $h(v)$ denote an underestimate to the optimal cost for the SPP-GCS from $v$ to $d$, $i.e.$, $h(v)\leq C_{opt}(v,d)$. We refer to such underestimates as admissible (just like in \astar). Note that for the destination $d$, the underestimate $h(d)=0$. We also refer to $h(.)$ as a heuristic function for \agcs.
Given any $S\subset V$, let $N_S$ denote the set of all the vertices in $V\setminus S$ that are adjacent to at least one vertex in $S$, $i.e.$ $N_S:=\{ v : u\in S, v\notin S, (u,v)\in E\}$. $N_S$ is also referred to as the neighborhood of $S$.

\begin{figure}
    \centering
\includegraphics[scale=0.35]{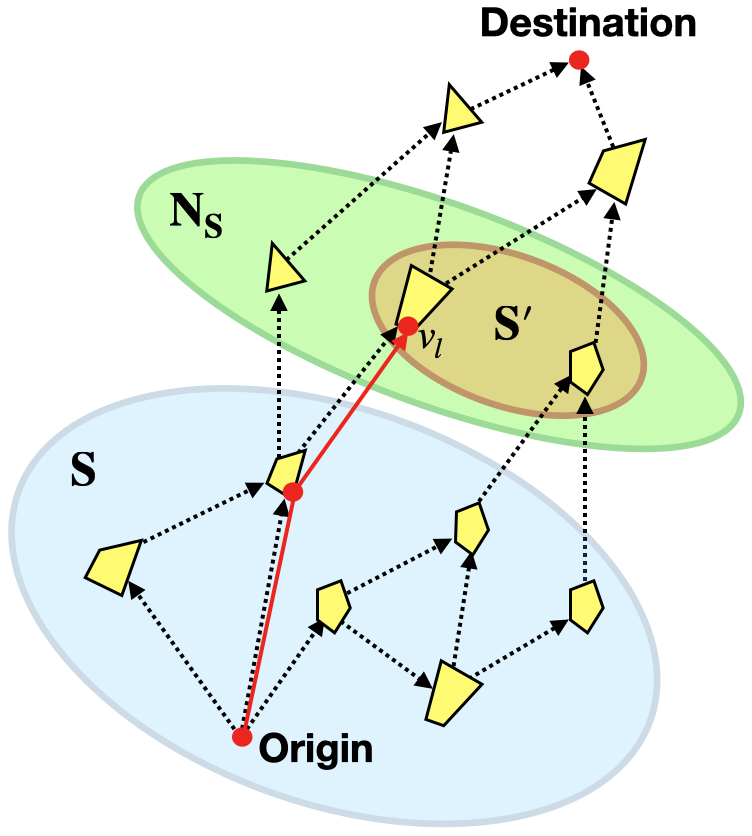}
    \caption{Setup in the SPP*-GCS showing the cut-set $S$, and subsets $N_S$ and $S'$. $N_S$ contains all the vertices in $V\setminus S$ that are adjacent to $S$, and $S'$ is any nonempty subset of $N_S$. A feasible path for SPP*-GCS is shown with solid (red) line segments.}
    \label{fig:GSPP}
\end{figure}
 
Now, consider a closely related problem to SPP-GCS referred to as SPP*-GCS defined on non-empty subsets $S\subset V$ and $S'\subseteq N_S$ such that $s\in S$ and $ d\notin S$. We will refer to any set $S\subset V$ that satisfies $s\in S$ and $ d\notin S$ as a {\it cut-set} in this paper (refer to Fig. \ref{fig:GSPP}). We will restrict our attention to the vertices in $\overline{V}:=S\bigcup S'$ and the edges in $\overline{E}:=\{(u,v): u\in S,v\in \overline{V},(u,v)\in E\}$. The objective of SPP*-GCS is to find a {terminal vertex} $v_l$, a path $\mathcal{P}_g:=(v_1,v_2,\cdots,v_l)$ in the graph $(\overline{V},\overline{E})$ and points $x_{v_i} \in {\mathcal{X}}_{v_i}$, $i=1,\cdots,l$  such that $v_1=s$, $v_l\in S'$, and the sum of the cost of traveling $\mathcal{P}_g$ and the heuristic cost, $h(v_l)$, is minimized. The optimal cost for SPP*-GCS is denoted as \boldmath $C^*_{opt}(S,S')$\unboldmath. The following Lemma shows the relationship between SPP*-GCS and SPP-GCS.

\begin{lemma}
    SPP*-GCS is a generalization of SPP-GCS.
\end{lemma}

\begin{proof}
Without loss of generality, we assume there is at least one feasible path from $s$ to $d$ in the graph $(V,E)$. Therefore, if $S:=V\setminus \{d\}$, then $N_S:=\{d\}$. Hence, $C^*_{opt}(V\setminus \{d\},\{d\})=C_{opt}(s,d)$ as $h(d)=0$.
\end{proof}

\subsection{Non-Linear Program for SPP*-GCS}
 
 To formulate SPP*-GCS, we use two sets of binary variables to specify the path and a set of continuous variables to choose the points from the convex sets corresponding to the vertices in the path. The first set of binary variables, denoted by $y_{uv}$ for $(u,v) \in \overline{E}$, determines whether the edge $(u,v)$ is selected in a solution to the SPP*-GCS. Here, $y_{uv} = 1$ implies that the edge $(u,v)$ is chosen, while $y_{uv} = 0$ implies the opposite. The second set of variables, denoted by $\alpha_v$ for $v \in S'$, determines whether the vertex $v$ is selected as a {terminal} or not. The continuous variable $x_{u}$ for $u\in \overline{V}$ specifies the point chosen in the convex set corresponding to vertex $u$. For any vertex $u\in \overline{V}$, let $\delta^+(u)$ be the subset of all the edges in $\overline{E}$ leaving $u$, and let $\delta^-(u)$ be the subset of all the edges in $\overline{E}$ coming into $u$. The non-linear program for SPP*-GCS is as follows:

 {\small
\begin{equation}\label{eq:SPP:obj}
C^*_{opt}(S,S') :=   \textrm{min}  \sum_{(u,v)\in \overline{E}} cost(x_u,x_v)y_{uv} + \sum_{v \in S'} \alpha_v h(v)
\end{equation}
\begin{align}
    \sum_{v\in S'}{\alpha_v } & =1  \label{eq:sumalpha}\\
    \sum_{v\in \overline{V}} y_{sv} & = 1 \label{eq:sdegree}\\
    \sum_{(u,v)\in \overline{E}} y_{uv} &= \alpha_v \textrm{ for }v\in S'  \label{eq:tdegree}
    \end{align}
    \begin{align}
    \sum_{(u,v)\in \delta^-{(v)}} y_{uv} &=\sum_{ (v,u)\in \delta^+{(v)}} y_{vu} \textrm{ for }v\in S\setminus \{s\}  \label{eq:idegree} \\
    x_u & \in \mathcal{X}_u\textrm{ for all } u\in \overline{V}  \label{eq:perspective}\\
    y_{uv} & = \{0,1\}\textrm{ for all } (u,v)\in \overline{E} \\ \alpha_v & =\{0,1\}\textrm{ for all } v\in S'  \label{eq:alphabinary}
\end{align}
}
Constraints \eqref{eq:sumalpha} state that exactly one vertex in $S'$ must be chosen as a {terminal}. Constraints \eqref{eq:sdegree}-\eqref{eq:idegree} state the standard flow constraints for a path from $s$ to a {terminal} in $S'$. Constraint \eqref{eq:perspective} states that for any vertex $u\in \overline{V}$, its corresponding point must belong to $\mathcal{X} _u$. The main challenge in solving the above formulation arises from the multiplication of the travel costs and the binary variables in \eqref{eq:SPP:obj}. To address this challenge, we follow the approach in \cite{GCS} and reformulate SPP*-GCS as a Mixed Integer Convex Program (MICP) in the next subsection.

\subsection{MICP for SPP*-GCS}

In this paper, as we are dealing with travel costs derived from Euclidean distances, the objective in \eqref{eq:SPP:obj} can be re-written as:
{\small
\begin{align*}
C^*_{opt}(S,S') & =  \textrm{min}  \sum_{(u,v)\in \overline{E}}   \lVert x_uy_{uv}-x_vy_{uv}\rVert_2  + \sum_{v \in S'} \alpha_v h(v).
\end{align*}
}
In \cite{GCS}, the bi-linear terms $x_uy_{uv}$ and $x_vy_{uv}$ in the objective above are replaced with new variables, and the constraints in \eqref{eq:idegree},\eqref{eq:perspective} are transformed to form a MICP. Before we present the MICP corresponding to SPP*-GCS, we first need to define the {\it perspective} of a set. The perspective of a compact, convex set $\mathcal{X} \subset \mathbb{R}^n$ is defined as $\overline{\mathcal{X}}:= \{(x,\lambda): \lambda\geq 0, x\in \lambda \mathcal{X}\}$. Now, let $x_uy_{uv} = z_{uv}$ and $x_vy_{uv} = z'_{uv}$ for all $(u,v)\in \overline{E}$. By following the same procedure outlined in \cite{GCS}, we derive a Mixed-Integer Convex Program (MICP) for the SPP*-GCS as follows:

{\small
\begin{equation}\label{eq:MICP:obj}
 \textrm{min}  \sum_{(u,v)\in \overline{E}} \lVert z_{uv}-z'_{uv}\rVert_2  + \sum_{v \in S'} \alpha_v h(v) 
\end{equation}

\begin{align}
  \sum_{v \in S'}{\alpha_v} & =1 \label{eq:alphasumMICP}\\
    \sum_{v\in \overline{V}} y_{sv} & = 1 \label{eq:sdegreeMICP}
\end{align}
\begin{align}
    \sum_{(u,v)\in \overline{E}} y_{uv} &= \alpha_v \textrm{ for }v\in S'\\
    \sum_{(u,v)\in \delta^-{(v)}} (z'_{uv},y_{uv}) &=\sum_{ (v,u)\in \delta^+{(v)}} (z_{vu},y_{vu}) \textrm{ for }v\in S\setminus \{s\} 
    \end{align}
    \begin{align}
    (z_{uv},y_{uv}) & \in \overline{\mathcal{X}}_u\textrm{ for all } u\in \overline{V}, (u,v)\in \overline{E} \\
    (z'_{uv},y_{uv}) & \in \overline{\mathcal{X}}_v\textrm{ for all } v\in \overline{V}, (u,v)\in \overline{E} \label{eq:perspectiveMICP}\\
    y_{uv} & = \{0,1\}\textrm{ for all } (u,v)\in \overline{E} \\ \alpha_v & =\{0,1\}\textrm{ for all } v\in S' \label{eq:alphaMICP}
\end{align}}

\begin{lemma}
    The MICP formulation in \eqref{eq:MICP:obj}-\eqref{eq:alphaMICP} for the SPP*-GCS has an optimal value equal to $C^*_{opt}(S,S')$.
\end{lemma}
\begin{proof}
    This result follows the same proof as Theorem 5.7 in \cite{GCS}.  
\end{proof}

\noindent We now arrive at the relaxation used in our algorithm:
{\small
\begin{tcolorbox}[colback=blue!5!white,colframe=blue!75!black,title={\bf Convex Relaxation for SPP*-GCS}]
\begin{equation*}\label{eq:Relax:obj}
\boldsymbol{R^*_{opt}(S,S')} = \textrm{min}  \sum_{(u,v)\in \overline{E}} \lVert z_{uv}-z'_{uv}\rVert_2  + \sum_{v \in S'} \alpha_v h(v) 
\end{equation*}
subject to the constraints in \eqref{eq:alphasumMICP}-\eqref{eq:perspectiveMICP} and,
\begin{align}
    0\leq y_{uv} \leq 1\textrm{ for all }(u,v)\in \overline{E}, \nonumber \\
    0\leq \alpha_v \leq 1\textrm{ for all }v \in S'. \label{eq:relax}
\end{align}
\end{tcolorbox}
}
\section{\bagcs}

\agcs (Algorithm \ref{Alg:AGCS}) starts with the input set $S := S_{\text{init}}$ and iteratively grows $S$ until further growth is not possible or is not beneficial. This set $S$ during any iteration of \agcs will always be a cut-set. \agcs keeps track of the growth of $S$ in two phases ({\it Phase 1} and {\it Phase 2}). The phases are determined based on the presence of the destination in the neighborhood of $S$. 

In {\it Phase 1}, $d$ is not a member of the neighborhood of $S$ (line \ref{whilePhase1} of Algorithm \ref{Alg:AGCS}). Therefore, in this phase, we can find lower bounds by solving the relaxation (line \ref{solveRelaxPhase1} of Algorithm \ref{Alg:AGCS}), but we cannot convert a relaxed solution to a feasible solution for SPP-GCS. Additionally, we add a vertex $v \in S'$ to $S$ if any edge $(u, v)$ leaving $S$ has $y_{uv} > 0$ (line \ref{updateSubsetPhase1} of Algorithm \ref{Alg:AGCS}, Algorithm \ref{Alg:UpdateSubset}).

In {\it Phase 2}, $d$ is part of the neighborhood of $S$, and therefore, it is possible to use the relaxation to find a lower bound and a feasible solution for SPP-GCS (lines \ref{solve1RelaxPhase2}-\ref{sol1FPhase2}, \ref{solve2RelaxPhase2}-\ref{sol2FPhase2} of Algorithm \ref{Alg:AGCS}). During this phase, we expand $S$ until one of the following two termination conditions is met: 1) all the vertices in $V\setminus \{d\}$ have already been added to $S$, or, effectively in this condition, $S$ cannot grow any further (line \ref{whilePhase2} of Algorithm \ref{Alg:AGCS}), 2) the relaxation cost of traveling through any vertex in $S':=N_S\setminus \{d\}$ becomes greater than or equal to the relaxation cost of traveling directly to $d$, or, effectively in this condition, growing $S$ further may not lead to better lower bounds (line \ref{boundPhase2} of Algorithm \ref{Alg:AGCS}). Also, the subroutine for expanding $S$ in {\it Phase 2} (line \ref{updateSetPhase2} of Algorithm \ref{Alg:AGCS}) mirrors that of {\it Phase 1}.

The following are the {\it key features} of \agcs:
\begin{itemize}
    \item The initial set $S_{\text{init}}$ can be any subset of $V$ as long as it is a cut-set. Specifically, $S_{\text{init}}$ can either consist solely of the origin (similar to how we initiate \astar for the SPP) or be generated by an algorithm. In this paper, we also consider $S_{\text{init}}$ generated by \astar as follows: Suppose $\bar{S}_{A^*}$ is the closed set (which also includes $d$, the last vertex added to $\bar{S}_{A^*}$) found by \astar when implemented on the centroids of all the convex sets associated with the vertices in $V$; let $S_{A^*}:=\bar{S}_{A^*}\setminus {d}$ and then assign $S_{\text{init}}:= S_{A^*}$. 
    \item \agcs can also be preemptively stopped at the end of any iteration. \agcs will always produce a lower bound at the end of any iteration in any phase and may produce a feasible solution during an iteration if $S$ is processed in {\it Phase 2}. 
    \item The algorithm for updating the set (Algorithm \ref{Alg:UpdateSubset}) based on the fractional values of the relaxed solutions is user-modifiable. Therefore, different variants of \agcs can be developed tailored to specific applications.
\end{itemize}

\begin{remark}
\agcs initiated with $S:=\{s\}$ mimics \astar in the special case when each of the convex sets is a singleton. Refer to the supplementary material for more details.
\end{remark}

\begin{remark}\label{rem:feasible}
While \agcs can be used to generate both bounds and feasible solutions, in this paper, we use \agcs to primarily find bounds. Also, we generate feasible solutions using the following two-step heuristic which performs reliably well for SPP-GCS: 1) Given any convex set, choose its centroid as its representative point, and run \astar on $(V,E)$ with the chosen points to find a path $\mathcal{P}_{A^*}$. 2) Let $S:= V\setminus \{d\}$ and assign $y_{uv}=1$ for each edge $(u,v)\in \mathcal{P}_{A^*}$ in the formulation  \eqref{eq:SPP:obj}-\eqref{eq:alphabinary}. Solve the resulting convex program to obtain the point corresponding to each vertex visited by $\mathcal{P}_{A^*}$.
\end{remark}

\section{Theoretical Results}

We will first demonstrate the termination of \agcs in a finite number of iterations and subsequently establish the validity of the bounds it generates. Without loss of generality, we assume there is at least one path from $s$ to $d$ in the input graph $G=(V,E)$. This assumption helps avoid certain trivial corner cases that may arise when updating the subsets (lines \ref{updateSubsetPhase1}, \ref{updateSetPhase2} of Algorithm \ref{Alg:AGCS}) or the bounds (lines \ref{updateLBPhase1}, 
 \ref{updateLBPhase2} of Algorithm \ref{Alg:AGCS}).


\begin{lemma}
    \agcs will terminate after completing at most $|V|-1$ iterations of both Phase 1 and Phase 2.
\end{lemma}

\begin{proof} If the algorithm enters {\it Phase 1}, each iteration of the phase must add at least one vertex from $N_S$ to $S$ since there is at least one edge $(u,v)$ leaving $S$ such that $y_{uv}>0$.  If the algorithm enters {\it Phase 2} and the termination condition in \ref{boundPhase2} is not met, $R^*_{opt}(S,S')$ is some finite value, and therefore, there is at least one edge $(u,v)$ leaving $S$ and entering $S'$ such that $y_{uv}>0$. In both the phases, at most $|V|-1$ vertices can be added to $S$ before one of the termination conditions (lines \ref{whilePhase2}, \ref{boundPhase2} of Algorithm \ref{Alg:AGCS}) become valid. Hence proved.
\end{proof}

\begin{algorithm}[t!]
{\small
	\textbf{Inputs:} \\
	$G=(V,E)$ \tcp{Input graph}
   $s,d\in V$ \tcp{Origin and destination}
   $\mathcal{X}_u ~ \forall u\in V$ \tcp{Input convex sets}
	${h}(v) \ \forall v \in V$ \tcp{admissible lower bounds} 
         $S_{\text{init}}\subset V$  \tcp{Input cut-set} 
	\textbf{Output:}  \\
	$C_{lb}$ \tcp{Lower bound for SPP-GCS} 
        $Sol_f$ \tcp{Feasible solution for SPP-GCS} 
	\textbf{Initialization:}  \\
        $S\leftarrow S_{\text{init}}$ \\
	$C_{lb} \leftarrow 0$ \\

    {\bf Main Loop:} \\

     {\bf \underline{Phase 1}} \\
     
	\While{$d \notin N_S$} 
 	{      \label{whilePhase1} $S'\leftarrow N_S$ \\
          $R^*_{opt}(S,S')$, $\mathcal{F}^*_{nd}$ $\leftarrow$ $SolveRelaxation(S,S')$ \label{solveRelaxPhase1} \\
          $C_{lb}\leftarrow \max(C_{lb},R^*_{opt}(S,S'))$ \label{updateLBPhase1}\\
          $S\leftarrow UpdateSubset(S,S',\mathcal{F}^*_{nd})$ \label{updateSubsetPhase1} \\
	}

   {\bf \underline{Phase 2}} \\

   $R^*_{opt}(S,\{d\})$, 
   $\mathcal{F}^*_d$ $\leftarrow$ $SolveRelaxation(S,\{d\})$ \label{solve1RelaxPhase2} \\
$Sol_f\leftarrow$$UpdateFeasibleSol(Sol_f,\mathcal{F}^*_d)$ \label{sol1FPhase2}\\

      \While{$\{d\} \neq N_S$}
 	{ \label{whilePhase2} $S'\leftarrow N_S\setminus \{d\}$ \\
  
           $R^*_{opt}(S,S')$, $\mathcal{F}^*_{nd}$ $\leftarrow$ $SolveRelaxation(S,S')$ \\
       
          $C_{lb}\leftarrow \max(C_{lb},\min(R^*_{opt}(S,S'),R^*_{opt}(S,\{d\})))$ \label{updateLBPhase2}\\
          $Sol_f\leftarrow$$UpdateFeasibleSol(Sol_f,Sol^*_d)$ \\

          \eIf{$R^*_{opt}(S,S')\geq  R^*_{opt}(S,\{d\})$}{ \label{boundPhase2}
          {\bf break} \\
          }
         {
          $S\leftarrow UpdateSubset(S,S',\mathcal{F}^*_{nd})$ \label{updateSetPhase2}\\
              $R^*_{opt}(S,\{d\})$, 
   $\mathcal{F}_d^*$ $\leftarrow$ $SolveRelaxation(S,\{d\})$\label{solve2RelaxPhase2} \\
$Sol_f\leftarrow$$UpdateFeasibleSol(Sol_f,\mathcal{F}^*_d)$ \label{sol2FPhase2}\\
          
         }
 	}
 	\textbf{return} $C_{lb}$, $Sol_f$
 	\caption{$A^*$-$GCS$}
 	\label{Alg:AGCS}
}
\end{algorithm}

\begin{algorithm}[t!]
{\small
    \SetAlgoLined
 Solve the relaxation in \eqref{eq:relax} given $S$ and $S'$ \\
 $R^*_{opt}(S,S') \leftarrow$ Optimal relaxation cost \\
 $\mathcal{F}^*\leftarrow$ Optimal solution to the relaxation \\
 	\textbf{return} $R^*_{opt}(S,S'),\mathcal{F}^*$
 	\caption{{SolveRelaxation}$(S,S')$}
 	\label{Alg:SolveRelaxation}
}
\end{algorithm}

\begin{algorithm}[t!]
{\small
    \SetAlgoLined
 \tcp{For any $(u,v)\in \overline{E}$, let the optimal value of $y_{uv}$ in $\mathcal{F}^*_{nd}$ be $y^*_{uv}$.}
       $O_S\leftarrow \{v: u\in S, v\in S', y^*_{uv}>0\}$  \\
          $S\leftarrow S\bigcup O_S$  \\
 	\textbf{return} $S$
 	\caption{{UpdateSubset}$(S,S',\mathcal{F}^*_{nd})$}
 	\label{Alg:UpdateSubset}
}
\end{algorithm}

\begin{algorithm}[t!]
{\small
    \SetAlgoLined
       Use randomized rounding \cite{GCS} or other heuristics to convert $\mathcal{F}^*_d$ into a feasible solution,  $Sol^*_f$. \\
        \lIf{$Cost(Sol^*_f)< Cost(Sol_f)$}
       {$Sol_f\leftarrow Sol^*_f$}
 	\textbf{return} $Sol_f$
 	\caption{{UpdateFeasible}$(Sol_f,\mathcal{F}^*_{d})$}
 	\label{Alg:updatefeasible}
  }
\end{algorithm}

\begin{lemma}\label{lemma1}
Consider any subset $S\subset V$ such that $s\in S$ and $d\notin S$. Let the path in an optimal solution to the SPP-GCS be denoted as $\mathcal{P}^*:=(v^*_1=s,v^*_2,\cdots,v^*_k=d)$ and let the optimal points in the corresponding convex sets be denoted as $x_{v^*_i}, i=1,\cdots,k$. For some $p\in \{2,\cdots,k\}$, let ${v^*_p}$ be such that ${v^*_i}\in S$ for $i=1,\cdots,p-1$ and ${v^*_p}\in N_S$. Let $S'$ be any subset of $N_S$ that contains ${v^*_p}$. Then, the optimal cost of the SPP-GCS is at least equal to the optimal relaxation cost of SPP*-GCS defined on $S$ and $S'$, $i.e.$, $ C_{opt}(s,d) \geq R^*_{opt}(S,S')$. 
\end{lemma}
\begin{proof}
     Now, the cost of the given optimal solution to SPP-GCS is:

    \begin{align}
        & C_{opt}(s,d)= \sum_{i=1}^{k-1} cost(x_{v^*_i},x_{v^*_{i+1}}) \nonumber \\
        &= \sum_{i=1}^{p-1} cost(x_{v^*_i},x_{v^*_{i+1}}) + \sum_{i=p}^{k-1} cost(x_{v^*_i},x_{v^*_{i+1}}) \nonumber \\
        &\geq \sum_{i=1}^{p-1} cost(x_{v^*_i},x_{v^*_{i+1}}) + h(v^*_{p)}.   \hspace{.2cm} \textrm{($\because$ $h(v^*_p)$ is admissible)} \label{eq:boundtheorem1} 
    \end{align}
Now, note that the path $(v^*_1=s,v^*_2,\cdots,v^*_p)$ and the points in the corresponding convex sets $x_{v^*_i}, i=1,\cdots,p$ is a feasible solution to SPP*-GCS defined on $S$ and $S'$. Therefore, the equation above \eqref{eq:boundtheorem1} further reduces to $C_{opt}(s,d) \geq C^*_{opt}(S,S') \geq R^*_{opt}(S,S')$. Hence proved. 
\end{proof}

\begin{theorem}\label{theorem1}
Consider any subset $S\subset V$ such that $s\in S$ and $d\notin S$. Then, $ C_{opt}(s,d) \geq R^*_{opt}(S,N_S)$. 
\end{theorem}

\begin{proof}
This theorem follows by applying Lemma \ref{lemma1} with $S':=N(S)$. Hence proved.
\end{proof}

\begin{theorem}\label{theorem2}
Consider any subset $S\subset V$ such that $s\in S$ and $d\in N_S $. Then, $ C_{opt}(s,d) \geq$ $ \min(R^*_{opt}$ $(S,N_S\setminus \{d\})$,$ R^*_{opt}(S,\{d\}))$. 
\end{theorem}

\begin{proof}
    Consider the edge $(v^*_{p-1},v_p^*)$ in the optimal path in Lemma \ref{lemma1} that leaves $S$ for the first time. Either $v_p^*=d$ or  $v_p^*\in N_S\setminus \{d\}$. If $v_p^*=d$, then applying Lemma \ref{lemma1} for $S'=\{d\}$ leads to $ C_{opt}(s,d) \geq R^*_{opt}(S,\{d\})$. On the other hand, if $v_p^*\in N_S\setminus \{d\}$, then applying Lemma \ref{lemma1} for $S'=N_S\setminus \{d\}$ leads to $ C_{opt}(s,d) \geq R^*_{opt}(S,N_S\setminus \{d\})$. Therefore, $ C_{opt}(s,d) \geq \min(R^*_{opt}(S,\{d\}), R^*_{opt}(S,N_S\setminus \{d\}))$. Hence, proved.
\end{proof}

\section{Numerical Results}

\noindent {\bf Set up:} We initially test our algorithm on graphs generated from mazes and axis-aligned bars in 2D (Fig. \ref{fig:instance-with-gcs}). Due to space constraints, the procedures for generating these graphs are given in the supplementary material. The convex sets in the GCS generated from any maze and axis-aligned instances are line segments (Fig. \ref{fig:mazeGCS}) and unit square regions (Fig. \ref{fig:barGCS}) respectively. The size of the GCS generated for each of the maps is provided in Table \ref{tab:small-instance-data}. The destination for each instance is always chosen to be the centroid of the top-rightmost square in the instance.

\begin{table}[h!]
    \small
    \centering
    \begin{tabular}{ccrrc}
        \toprule
        Map No. & Type & $|V|$ & $|E|$ \\ 
        \midrule
        \begin{filecontents*}{small_instance_data.csv}
ID,type,num_vertices,num_edges,size,name
1,Maze,121,425,small,maze_2d_p10
2,Maze,415,1020,small,maze_2d_p20
3,Maze,1615,3649,small,maze_2d_p40
4,Maze,6417,14239,small,maze_2d_p80
5,Bars,769,6525,small,structuredoxes_2d_p20
6,Bars,1790,15912,small,structuredoxes_2d_p30
\end{filecontents*}
\csvreader[late after line=\\]{small_instance_data.csv}{1=\one,2=\two,3=\three,4=\four,5=\five}{\one & \two & \three & \four}
        \bottomrule
    \end{tabular}
\caption{ 2D instances generated from mazes and bars.}
\label{tab:small-instance-data}
\end{table}

{We will compare the performance of the lower bounds generated by \agcs with the lower bound generated by the baseline algorithm that solves the convex relaxation of SPP-GCS on the entire graph. The upper bound\footnote{We did not use the rounding heuristic in \cite{GCS} to generate upper bounds because its performance, especially for maps 5, 6 and the larger 2D/3D maps, was significantly worse compared to the bounds produced by the two-step heuristic.} for both algorithms is the length of the feasible solution generated by the two-step heuristic.} Note that the size of the cut-set during any iteration of \agcs determines the sizes of the convex relaxations we will solve. Therefore, we will keep track of the sizes of the cut-sets in addition to the run times of the algorithms. Also, the baseline is equivalent to implementing \eqref{eq:relax} with $S:=V\setminus \{d\}$; therefore, the size of the cut-set corresponding to the baseline is $|V|-1$.  All the algorithms were implemented using the Julia programming language \cite{Julia-2017} and run on an Intel Haswell 2.6 GHz, 32 GB, 20-core machine. Gurobi \cite{gurobi} was used to solve all the convex relaxations. 

\noindent{\bf Heuristics for generating underestimates:} {We designed two heuristics to generate underestimates for \agcs. Given a set, the first heuristic ($h_1$) computes the shortest Euclidean distance between any point in the set and the destination, ignoring all other vertices/edges in the graph. $h_1$ is very fast and requires negligible computation time; it is by default present in \agcs and is already part of all the run times reported in the results section. While $h_1$ finds consistent underestimates, its bounds may not be strong.

The second heuristic ($h_2$), referred to as the {\it expand and freeze} heuristic, computes underestimates of the optimum through an iterative process. It starts from the destination and proceeds in the reverse direction, adding a subset of vertices $U$ in each iteration to the set containing the destination while determining the underestimate for each vertex in $U$. Specifically, in each iteration, the subset $U$ and its corresponding underestimates are obtained by solving a convex relaxation of SPP*-GCS. Any vertex that is incident to any fraction of an edge in the relaxed solution is included in $U$.} To limit the size of the relaxations, as soon as the size of the set containing the destination reaches a limit (say $n_{max}$), we shrink all the vertices in the set into a new destination vertex and repeat the process until all the (reachable) vertices are visited. In this paper, we set $n_{max}$ as 100. Generally, for our maps, $h_2$ produced tighter estimates than $h_1$, but $h_2$ can be inconsistent. {$h_2$ can run either in the order of seconds or minutes, depending on the map. We specify its computation times for all the maps in each of the results sections.} We also study the impact of using a convex combination of the two heuristics, $i.e.$, for any vertex $v\in V$, we define the combined heuristic estimate as $h(v):=(1-w)h_1(v) + wh_2(v)$ where $w$ is the weighting factor. 

\begin{figure}[tb!]
\centering
  \begin{subfigure}[t]{1\linewidth}
  \centering
    \begin{subfigure}[t]{.3\linewidth}
        \centering\includegraphics[width=1\linewidth]{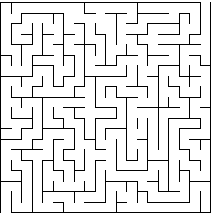}
    \end{subfigure}
    \begin{subfigure}[t]{.3\linewidth}
        \centering\includegraphics[width=1\linewidth]{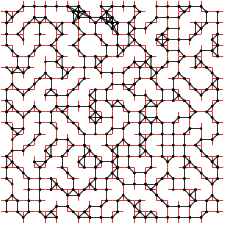}
  \end{subfigure} 
  \caption{Maze and its GCS}
   \label{fig:mazeGCS}
  \end{subfigure}
  \\
  \begin{subfigure}[t]{1\linewidth}
  \centering
    \begin{subfigure}[t]{.3\linewidth}
        \centering\includegraphics[width=1\linewidth]{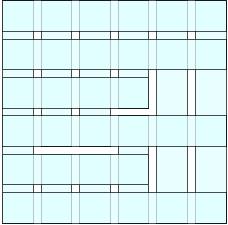}
    \end{subfigure}
    \begin{subfigure}[t]{.3\linewidth}
        \centering\includegraphics[width=1\linewidth]{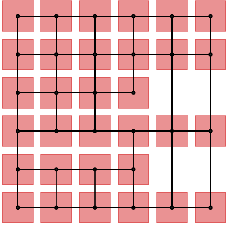}
  \end{subfigure} 
  \caption{Axis-aligned bars and its GCS}
  \label{fig:barGCS}
  \end{subfigure}
  \caption{2D maps and their GCS. In the GCS corresponding to the Maze, the red segments show the convex sets and the black lines show the edges. Similarly, in the GCS corresponding to the axis-aligned bars, the shaded squares show the convex sets and the black lines show the edges.}
  \label{fig:instance-with-gcs}
\end{figure}

\begin{figure*}[htb!]
\centering
  \begin{subfigure}[t]{.3\linewidth}
    \centering\includegraphics[width=0.65\linewidth]{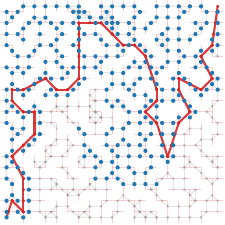}
    \caption{$|S_{A^*}|=288, w = 0.0$}
  \end{subfigure} ~~
  \begin{subfigure}[t]{.3\linewidth}
    \centering\includegraphics[width=0.65\linewidth]{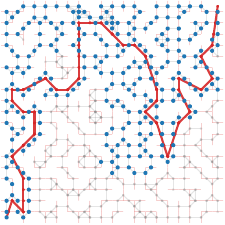}
    \caption{$|S_{A^*}|=266, w = 0.5$}
  \end{subfigure} ~~
  \begin{subfigure}[t]{0.3\linewidth}
    \centering\includegraphics[width=0.65\linewidth]{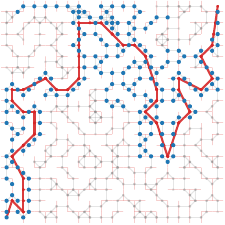}
    \caption{$|S_{A^*}|=196$, $w = 1.0$}
  \end{subfigure}
  \caption{{The feasible path (thick red lines) and the cut-set vertices (shown in color blue) produced by the two-step $A^*$ based heuristic for different weights.} There are 415 vertices in this GCS corresponding to map no. 2. \agcs in this case terminated in 1 iteration for all the weights and produced the same bounds. By first finding the cut-set ($S_{A^*}$) and then applying \agcs with $S_{init}:=S_{A^*}$, we reduced the size of the relaxation with $|V|-1=414$ vertices into a relaxation with $|S_{A^*}|$ vertices.}
  \label{fig:maze-solutions}
\end{figure*}

\noindent{\bf \normalsize Comparisons based on number of iterations and the sizes of the cut-sets for $\boldsymbol{\mathit{A^*}}$-$\boldsymbol{\mathit{GCS}}$  variants:} 
We compare the performance of \agcs for two choices of $S_\text{init}$: one with $S_\text{init}:=\{s\}$ and another with $S_\text{init}:=S_{A^*}$ (refer to the key features of \agcs on page 4). We fix the origin for each map in Table \ref{tab:small-instance-data} at the centroid of the left-bottommost unit square and implement the variants of \agcs on the corresponding graphs for heuristic weight  $w=1$. Both choices of $S_\text{init}$ produced the same bounds for all the six maps. Therefore, for both the variants, we compare the total number of iterations for both phases of \agcs, the size of the cut-set ($|S|$) upon termination of \agcs, and the total run time (in secs.) as shown in Table \ref{tab:small-instance-iter-start}. While $|S|$ for both variants of \agcs are relatively closer, \agcs initialized with $S_\text{init}:=\{s\}$ required significantly more iterations and run time to terminate compared to the \agcs initialized with $S_\text{init}:= S_{A^*}$. This trend remained the same for any other choice of origin in these maps. Therefore, for the remaining results in this paper, we will only focus on the variant of \agcs initialized with $S_\text{init}:=S_{A^*}$. Note that in this variant, the algorithm will never enter Phase 1 since the destination $d$ is already present in the neighborhood of $S_{A^*}$. Henceforth, the number of iterations of \agcs will simply refer to the number of iterations in Phase 2 of \agcs.

\begin{table}[!htb]
    \centering
    \small
    \footnotesize
    \begin{tabular}{ccccccc}
         \toprule 
         \multirow{2}{*}{Map No.} & \multicolumn{3}{c}{$S_\text{init}:=\{s\}$} & \multicolumn{3}{c}{$S_\text{init}:=S_{A^*}$} \\
         \cmidrule{2-7}
         & $n_{iter}$ & $|S|$ & time & $n_{iter}$ & $|S|$ & time \\
         \midrule 
         \begin{filecontents*}{set_vs_source.csv}
ID,5setiter,5setS,5settime,5sourceiter,5sourceS,5sourcetime,10setiter,10setS,10settime,10sourceiter,10sourceS,10sourcetime,20setiter,20setS,20settime,20sourceiter,20sourceS,20sourcetime,50setiter,50setS,50settime,50sourceiter,50sourceS,50sourcetime,100setiter,100setS,100settime,100sourceiter,100sourceS,100sourcetime,250setiter,250setS,250settime,250sourceiter,250sourceS,250sourcetime
1,5,115,0.29,42,115,1.1,4,95,0.23,38,91,0.84,3,66,0.11,33,65,0.52,3,46,0.08,27,38,0.25,3,44,0.07,27,36,0.23,3,44,0.07,27,36,0.23
2,2,275,0.27,140,246,6.86,1,240,0.11,130,210,5.72,1,222,0.12,124,178,4.92,1,204,0.1,117,167,4.18,1,196,0.1,117,149,3.4,1,195,0.1,115,143,3.07
3,2,622,1.02,400,567,58.91,4,577,1.51,394,538,54.82,2,537,0.87,383,474,47.98,2,520,0.76,383,433,43.01,2,518,0.69,383,422,43.98,2,517,0.71,384,428,46.13
4,2,4497,5.16,1796,4362,1995.62,2,4463,6.69,1771,4141,1802.82,2,4149,5.33,1791,3798,1572.28,2,4053,4.98,1736,3017,1228.91,2,4003,4.48,1726,2699,1077.99,2,3989,4.56,1749,2557,1064.98
5,1,767,4.07,128,767,185.79,2,767,8.35,122,764,178.1,1,738,2.92,117,723,159.12,1,711,4.37,111,675,137.9,1,654,2.8,110,585,126.3,1,544,2.64,84,414,61.99
6,2,1772,47.91,171,1766,1027.84,2,1775,65.05,174,1768,1114.34,1,1745,26.33,172,1755,1069.11,2,1748,62.33,173,1746,1211.56,1,1693,24.06,171,1657,962.94,1,1580,26.5,165,1531,820.01
\end{filecontents*}
\csvreader[late after line=\\]{set_vs_source.csv}{1=\one,29=\ten,31=\onetwo,30=\onethree,26=\onenine,28=\twoone,27=\twotwo}{\one & \ten & \onethree & \onetwo & \onenine & \twotwo & \twoone} 
         \bottomrule
    \end{tabular}
\caption{ No. of iterations ($n_{iter}$), the size of the cut-set ($|S|$) upon termination and the run time (in secs.) for the \agcs variants.}
\label{tab:small-instance-iter-start}
\end{table}

\begin{table}[!htb]
    \centering
    \footnotesize
    \begin{tabular}{ccccccc}
         \multicolumn{6}{c}{} \\
         \cmidrule{1-7}
           Map No. & 1 & 2 & 3 & 4 & 5 & 6\\
         \midrule 
          Time &14.4 & 21.5 & 48.4 & 151.7 & 43.3 & 94.4 \\
         \bottomrule
    \end{tabular}
\caption{Heuristic $h_2$ run time in secs.}
\label{tab:h2runtimes}
\end{table}

\begin{table*}[t!]
    \centering
    \footnotesize
    {\scriptsize
    \begin{tabular}{ccrrrrrrrrrrrrrrr}
         \toprule 
         \multirow{2}{*}{Map No.} & \multirow{2}{*}{Algorithm} & \multicolumn{3}{c}{$w=0$} & \multicolumn{3}{c}{$w=0.25$} & \multicolumn{3}{c}{$w=0.5$} & \multicolumn{3}{c}{$w=0.75$} & \multicolumn{3}{c}{$w=1.0$} \\
         \cmidrule{3-17}
         & & $|S|$ & time & gap & $|S|$ & time & gap & $|S|$ & time & gap & $|S|$ & time & gap & $|S|$ & time & gap \\ 
         \midrule 
         \begin{filecontents*}{small_stat.csv}
ID,Algorithm,0_nodes,0_time,0_gap,25_nodes,25_time,25_gap,50_nodes,50_time,50_gap,75_nodes,75_time,75_gap,100_nodes,100_time,100_gap,flag
,Baseline,120.0,0.1,0.1,120.0,0.1,0.1,120.0,0.1,0.1,120.0,0.1,0.1,120.0,0.1,0.1,
1,$A^*$-$GCS_{\infty}$,37.3,0.1,0.1,35.4,0.1,0.1,32.8,0.0,0.1,31.1,0.0,0.1,30.4,0.0,0.1,
,$A^*$-$GCS_{1}$,36.1,0.0,3.4,34.1,0.0,3.4,31.6,0.0,3.4,29.8,0.0,3.4,28.9,0.0,3.4,X
,Baseline,414.0,0.4,0.3,414.0,0.4,0.3,414.0,0.4,0.3,414.0,0.3,0.3,414.0,0.3,0.3,
2,$A^*$-$GCS_{\infty}$,225.9,0.2,0.3,201.2,0.2,0.3,174.5,0.1,0.3,151.3,0.1,0.3,123.9,0.1,0.3,
,$A^*$-$GCS_{1}$,225.9,0.2,0.5,201.2,0.1,0.4,174.5,0.1,0.3,151.2,0.1,0.4,123.9,0.1,0.3,X
,Baseline,1614.0,1.0,0.0,1614.0,1.0,0.0,1614.0,1.0,0.0,1614.0,1.0,0.0,1614.0,1.0,0.0,
3,$A^*$-$GCS_{\infty}$,683.6,1.0,0.0,610.5,0.8,0.0,531.5,0.7,0.0,467.1,0.6,0.0,402.9,0.5,0.0,
,$A^*$-$GCS_{1}$,683.6,0.5,0.3,610.5,0.4,0.2,531.5,0.3,0.1,467.1,0.3,0.1,402.9,0.2,0.1,X
,Baseline,6416.0,5.9,0.0,6416.0,5.9,0.0,6416.0,5.9,0.0,6416.0,5.8,0.0,6416.0,5.9,0.0,
4,$A^*$-$GCS_{\infty}$,2858.2,3.4,0.0,2575.9,3.0,0.0,2309.3,2.7,0.0,2054.3,2.3,0.0,1786.8,1.9,0.0,
,$A^*$-$GCS_{1}$,2858.2,1.7,0.0,2575.9,1.4,0.0,2309.3,1.3,0.0,2054.3,1.1,0.0,1786.8,1.0,0.0,X
,Baseline,768.0,3.8,2.6,768.0,3.8,2.6,768.0,3.9,2.6,768.0,3.9,2.6,768.0,3.8,2.6,
5,$A^*$-$GCS_{\infty}$,343.9,2.8,2.6,314.6,2.5,2.6,291.7,2.1,2.5,265.3,1.8,2.6,241.2,1.3,2.6,
,$A^*$-$GCS_{1}$,340.9,1.4,3.3,312.8,1.3,3.4,290.3,1.2,3.0,264.6,1.2,3.0,240.5,1.0,2.8,X
,Baseline,1789.0,20.5,2.4,1789.0,20.6,2.4,1789.0,20.5,2.4,1789.0,20.6,2.4,1789.0,20.6,2.4,
6,$A^*$-$GCS_{\infty}$,911.2,24.0,2.4,809.3,19.4,2.4,713.3,13.3,2.4,616.8,10.2,2.4,515.1,6.4,2.4,
,$A^*$-$GCS_{1}$,882.9,9.4,4.5,788.4,7.1,3.9,694.4,6.0,3.7,601.2,3.6,4.2,508.7,2.8,2.9,
\end{filecontents*}
\csvreader[head to column names, 
         after line=\ifthenelse{\equal{\flag}{}}{\\}{\\\midrule}, before first line=, table foot=\\\bottomrule,]{small_stat.csv}{1=\one,2=\two,3=\three,4=\four,5=\five,6=\six,7=\seven,8=\eight,9=\nine,10=\ten,11=\eleven,12=\twelve,13=\thirteen,14=\fourteen,15=\fifteen,16=\sixteen,17=\seventeen}{\one & \two & \three & \four & \five & \six & \seven & \eight & \nine & \ten & \eleven & \twelve & \thirteen & \fourteen & \fifteen & \sixteen & \seventeen} \\[-2ex]\bottomrule
    \end{tabular}}
    \caption{ Average size of the cut-set ($|S|$), average run times (in secs) and average optimality gap (in \%) for different levels of heuristic weights for maps in Table \ref{tab:small-instance-data}. }
\label{tab:small-stat}
\end{table*}

\noindent{\bf Impact of the quality of underestimates:} 
For each map in Table \ref{tab:small-instance-data}, we generated 100 instances, with each instance choosing a different origin for the path. {The run times for heuristic $h_2$ are shown in Table \ref{tab:h2runtimes}. Note that this heuristic is run only once and is used as an input for all the 100 instances for any heuristic weight $w>0$.} Given an instance $I$, its optimality gap is defined as $100 \times \frac{C_f - C_{lb}}{C_{lb}}$, where $C_f$ is the length of the feasible solution obtained by the two-step algorithm in Remark \ref{rem:feasible}, and $C_{lb}$ is the bound returned by \agcs. In Table \ref{tab:small-stat}, we present the results (average size of cut-sets, average run times in seconds, and the average optimality gap) at the end of the first iteration of \agcs (referred to as \agcsi) as well as upon the termination of \agcs (referred to as \agcsf). Even with a heuristic weight $w=0$ (which corresponds to the first heuristic $h_1$ that provides relatively weak bounds), \agcs is able to reduce the size of the cut-sets (which, in turn, {\it reduces the sizes of convex relaxations to solve}) by more than 45\%, on average (except in map no. 6). In fact, by selecting the cut-set informed by \astar with $w=0$ and completing just one iteration of \agcs, we achieve optimality gaps in a shorter time that closely match those provided by the baseline algorithm (except for map no. 6). Refer to Fig. \ref{fig:maze-solutions} for an illustration. For map. no. 6 the underestimates provided by $h_1$ was much weaker compared to $h_2$, and hence \agcs performed significantly better when the heuristic weight was increased to 1. 

{For single query problems, the results for the 2D maps show that the default heuristic ($h_1$) in \agcs is sufficient to provide benefits in terms of reduction in problem sizes and computation time. On the other hand, for multi-query problems, using $h_2$ can further significantly reduce the computation time while providing similar optimality gaps. For example, in map no. 6, \agcsi is at least six times faster than the baseline algorithm for $w=1$. \agcs performed the best in map no. 4 for all heuristic weights; specifically, \agcsi found optimal solutions with average run times considerably lower compared to the baseline for this map.} As the heuristic weight increased, the size of the cut-sets as well as the run times decreased on average. 

\begin{figure}[t!]
\centering\includegraphics[scale=0.16]{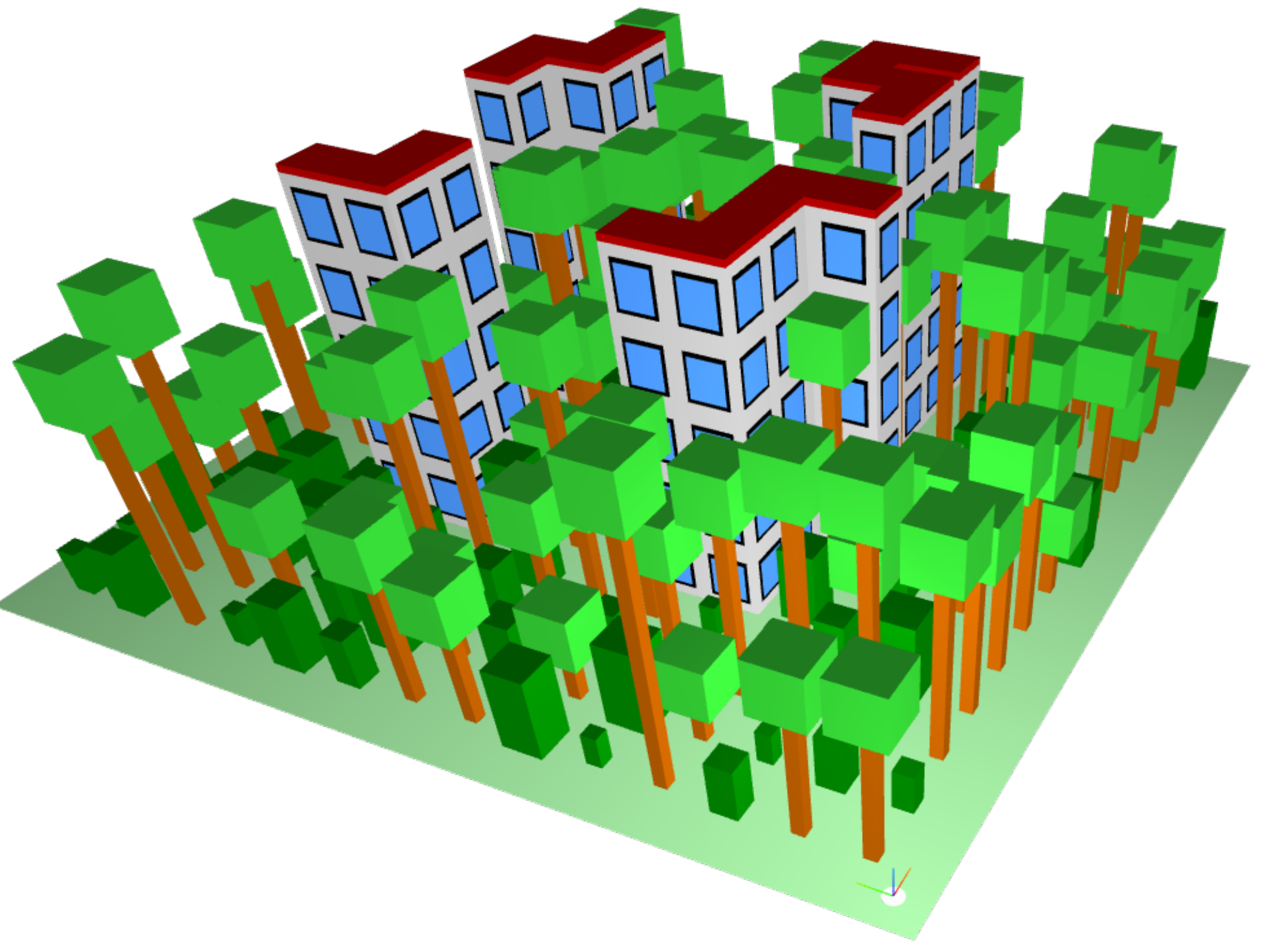}
    \caption{3D village map.}
    \label{fig:villagemap}
\end{figure}

\begin{figure}[t!]
    \centering
    \includegraphics[width=.8\linewidth]{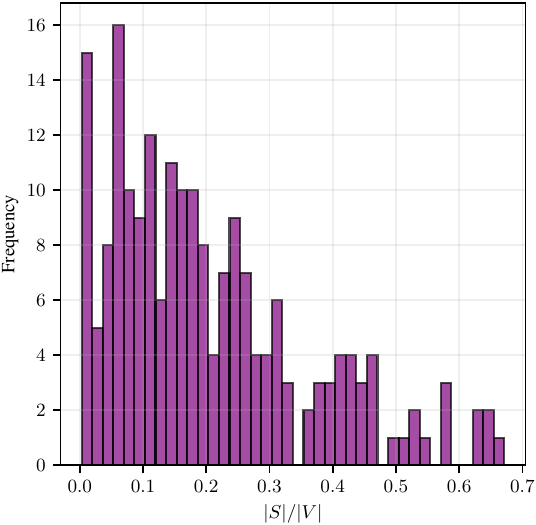}
    \caption{Histogram showing the empirical distribution of the fraction of vertices solved by \agcsf over all the 200 instances of the 3D maps. Here, $|S|$ denotes the size of the cut-set and $|V|$ represents the total number of vertices in the graph.}
    \label{fig:histogram}
\end{figure}

\noindent{\bf $\boldsymbol{\mathit{A^*}}$-$\boldsymbol{\mathit{GCS}}$  on large 3D village maps:} Next, \agcs was tested on two large 3D village maps (Fig. \ref{fig:villagemap}) created using the meshcat library in Python. We followed a procedure similar to the one described for maps in Table \ref{tab:small-instance-data} to develop the GCS for these village maps. The convex sets in these maps, resulting from the intersection of adjacent cubes, are axis-aligned squares (shared sides of adjacent cubes) computed using the algorithm in \cite{zomorodian2000fast}. The GCS corresponding to the first 3D map has 6010 vertices and 166,992 edges, while the GCS for the second 3D map has 11,258 vertices and 321,603 edges. In both maps, the destination was chosen at one of the corners of the 3D region. The computation time for heuristic $h_2$ for these two maps was 53.30 secs and 89.10 secs respectively.

We generated 100 instances for each 3D map, varying the origins of the paths. Across all 200 instances, \agcsf produced approximately the same optimality gaps as the baseline but, on average, explored only 20.38\% of all vertices in the graphs. The histogram showing the fraction of vertices solved by the convex program using \agcsf relative to the total number of vertices for all instances is presented in Fig \ref{fig:histogram}. Specifically, for 63\% of the first 3D map instances, \agcsf ran faster than the baseline, with an average savings of 334.27 seconds (around a 61\% reduction). Similarly, for 63\% of the second 3D map instances, \agcsf ran faster than the baseline, with an average savings of 904.49 seconds (around a 65\% reduction). For the remaining instances, even though \agcsf required more time, \agcsi produced optimality gaps comparable to the baseline with smaller computation times ({\it refer to the supplementary material}).

\section{Conclusions}

In this paper, we show how to effectively combine the existing convex-programming based approach with heuristic information to obtain near-optimal solutions for SPP-GCS. While our results clearly illustrate the benefits of \agcs for travel costs based on Euclidean distances, we do not claim that the proposed approach is superior to directly solving the relaxation of SPP-GCS for every instance. Nevertheless, the proposed approach is the first of its kind for solving convex relaxations of SPP-GCS and opens a new avenue for research into related problems. 

 \section{Acknowledgment}
    This material is based upon work partially supported by the National Science Foundation under Grant No. 2120219. Any opinions, findings, and conclusions or recommendations expressed in this material are those of the author(s) and do not necessarily reflect the views of the National Science Foundation.

\bibliography{references}

\section{Supplementary Material}

\subsection{Special case of SPP-GCS}\label{specialcase}

\begin{remark}\label{remark1}
    In the special case when the convex sets reduce to singletons, SPP*-GCS simplifies to a variant of a SPP where the objective is to find a path from the origin to a vertex $v_l$ in $S'$ such that the sum of the cost to arrive at $v_l$ and the heuristic cost $h(v_l)$ is minimized. This variant of SPP is exactly the optimization problem solved during each iteration of \astar, where $S$ represents the closed set and $S'$ denotes the open set. 
\end{remark}
\begin{remark}\label{remark2}
 In the special case when the convex sets reduce to singletons, the relaxation of SPP*-GCS in \eqref{eq:relax} can be re-formulated as a linear program with a coefficient matrix that is Totally Uni-Modular (TUM) \cite{lawler:76}, and as a result, its extreme points are optimal solutions for the SPP*-GCS.    
\end{remark}

Assume each of the convex sets associated with the vertices in $V$ is a singleton. Let the cost of traveling from $u\in V$ to $v\in V$ be denoted as $cost_{uv}\in \mathbb{R}^+$. Let the underestimates denoted by $h(u),~ u\in V$ be consistent, $i.e.$, $h(u)\leq cost_{uv} + h(v)$ for all $(u,v)\in E$. In this special case, consider an implementation of \astar on $(V,E)$ with its closed set initialized to $\{s\}$. Until termination, each iteration of \astar picks a vertex (say $v^*$) in $N_S$ with the least $f$ cost\footnote{Here, we borrow the usual definition of $f$ and $g$ costs from A* \cite{hart1968aFormalBasis}.} for expansion and moves $v^*$ from $N_S$ to $S$. If the underestimates are consistent, once $v^*$ is chosen for expansion and added to $S$, it will never be chosen for expansion again ($i.e.$, never removed from $S$ again). \astar terminates when $v^*$ is the destination. 
This is also how \agcs performs in this special case for the following reasons: From Remarks \ref{remark1}-\ref{remark2}, solving the relaxations in either Phase 1 or Phase 2 of \agcs is equivalent to solving a shortest path problem where the objective is to choose a vertex $v$ in $N_S$ with the least $f$ cost (the sum of the travel cost from $s$ to $v$ and $h({v})$) such that the path starts from the origin, travels through the vertices in $S$ and reaches one of the vertices in $N_S$. Until termination, each iteration of any of the phases in \agcs will move one vertex from $N_S$ to $S$ (based on where the shortest path ends in $N_S$). In addition, if the shortest path found in Phase 2 of \agcs ends at the destination (or if $R^*_{opt}(S,S')\geq  R^*_{opt}(S,\{d\})$ is true), \agcs will terminate just like \astar.

\subsection{GCS generation procedure}

To generate a GCS corresponding to a maze, we first partition its free space into unit squares and consider the obstacle-free sides shared by any two adjacent squares as vertices in our GCS. The convex sets here are the line segments corresponding to the shared sides. Two vertices (or shared sides) in a maze-GCS are connected by an edge if the shared sides belong to the same unit square. An example of a GCS corresponding to a maze is shown in Fig. \ref{fig:mazeGCS}. In the case of an instance with randomly generated, axis-aligned bars, we partition all the bars into unit squares and connect any two unit squares with an edge if they both belong to the same bar. Here, the convex sets are square regions, and an example of a GCS corresponding to a map with bars is shown in Fig. \ref{fig:barGCS}.

\subsection{\bf $\boldsymbol{\mathit{A^*}}$-$\boldsymbol{\mathit{GCS}}$ on a large 2D map} \agcs was also tested on a large 2D maze with its corresponding GCS containing 25,615 vertices and 56,550 edges. Similar to the previous runs, we generated 100 instances, varying the origins of the paths. {For all the instances, \agcs terminated in just one iteration for all the heuristic weights and produced optimal solutions. Specifically, \agcs with the default heuristic ($h_1$) outperformed the baseline both in terms of reduced problem sizes (cut-sets) and run time by more than 50\% (Fig. \ref{fig:large-maze-stat}). Heuristic $h_2$ required 330.25 secs to compute for this map. In a multi-query setting, using \agcs with $h_2$ can further provide significant benefits as seen in Fig. \ref{fig:large-maze-stat}.}

\begin{figure}
    \centering
    \includegraphics[scale=.8]{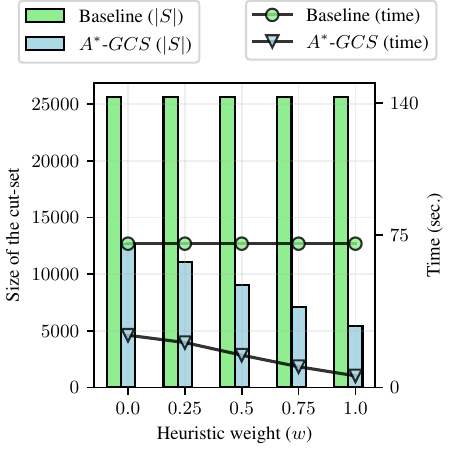}
    \caption{Average cut-set sizes and run times (in secs.) for a large 2D map. {Both the algorithms found the optimal solutions for all the instances in this map.} }
    \label{fig:large-maze-stat}
\end{figure}

\begin{figure*}[tb!]
    \centering 
    \includegraphics[scale=.95]{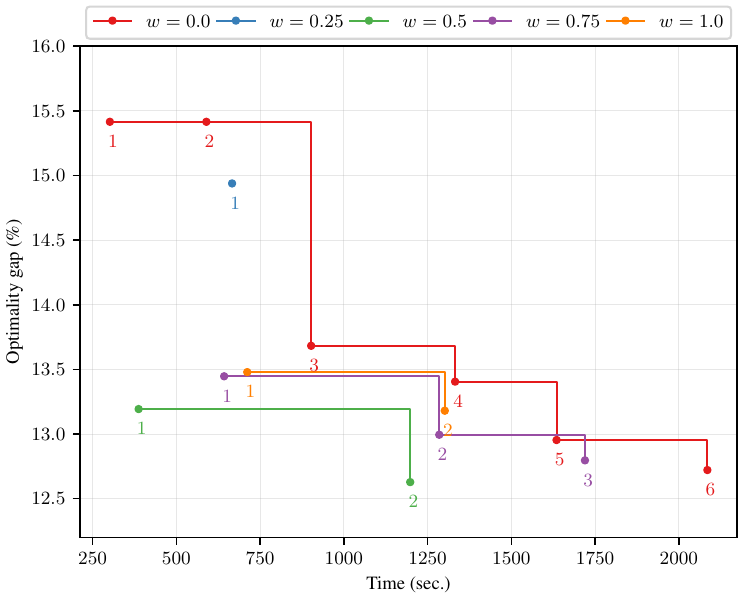}
    \caption{Results for a first 3D map instance obtained after each iteration of \agcs with varying weights. The numbers marked next to the lines indicate when each of the iterations of \agcs completed.  }
    \label{fig:maze-trace}
\end{figure*}

\subsection{$\boldsymbol{\mathit{A^*}}$-$\boldsymbol{\mathit{GCS}}$ on a large 3D map} 

Here, we explore an instance of the first 3D map where the computation times for \agcsf is larger than the baseline. The baseline algorithm for this instance required 1848 seconds to solve with an optimality gap of 12.56\%. The optimality gaps obtained after each iteration of \agcs and their corresponding run times are shown in Fig. \ref{fig:maze-trace}. {\it The key insight from this figure is that \agcsi is able to achieve an optimality gap of around 15.5\% in less than 300 seconds for heuristic weight $w=0$ (corresponding to the heuristic $h_1$).} Note that running \agcs to termination may not be helpful in this map, as the run times per iteration are relatively high; therefore, for heuristic weight $w=0$, \agcsf requires around 2000 seconds to complete. Even if one adds the $h_2$ computation time (53.3 secs) to the run time for any weight $> 0$, \agcsi outperforms the baseline in terms of run times while providing similar optimality gaps.

\end{document}